\documentclass{amsart}
\usepackage[utf8]{inputenc}
\usepackage[margin=1.4 in]{geometry}
\usepackage{amssymb,bm,amsfonts, stmaryrd, amsmath, amsthm, enumerate, mathtools,comment,color,xcolor}
\usepackage[all]{xy}
\makeatletter
\@namedef{subjclassname@2020}{\textup{2020} Mathematics Subject Classification}
\makeatother

\definecolor{chianti}{rgb}{0.6,0,0}
\definecolor{meretale}{rgb}{0,0,.6}
\definecolor{leaf}{rgb}{0,.35,0}
\usepackage[colorlinks=true,linkcolor=meretale,
urlcolor=chianti, citecolor=leaf,pagebackref]{hyperref}
\usepackage{wrapfig,tikz, tikz-cd}
\usetikzlibrary{arrows}
\usepackage[capitalise]{cleveref}
\crefformat{equation}{(#2#1#3)}
\crefrangeformat{equation}{(#3#1#4--#5#2#6)}
\crefformat{enumi}{(#2#1#3)}
\crefrangeformat{enumi}{(#3#1#4--#5#2#6)}
\newtheorem{theorem}{Theorem}[section]
\newtheorem{lemma}[theorem]{Lemma}
\newtheorem{proposition}[theorem]{Proposition}
\newtheorem{corollary}[theorem]{Corollary}
\newcounter{intro}

\newtheorem{introthm}[intro]{Theorem}
\newtheorem{introcor}[intro]{Corollary}

\theoremstyle{definition}

\newtheorem{example}[theorem]{Example}

\newtheorem{setup}[theorem]{Setup}

\newtheorem{chunk}[theorem]{}

\DeclareRobustCommand\longtwoheadrightarrow
     {\relbar\joinrel\twoheadrightarrow}

\theoremstyle{definition}

\newcommand{\pdim}{{\operatorname{pdim}}}

\newcommand{\Spec}{{\operatorname{Spec}}}

\newcommand{\Tor}{{\operatorname{Tor}}}

\renewcommand{\H}{\operatorname{H}}

\DeclareMathOperator{\thick}{\mathsf{thick}}

\DeclareMathOperator{\codepth}{codepth}

\DeclareMathOperator{\depth}{depth}

\newcommand{\V}{{\rm{V}}}

\newcommand{\lotimes}{\otimes^{\sf L}}
\DeclareMathOperator{\D}{\mathsf{D}}

\DeclareMathOperator{\reg}{reg}
\DeclareMathOperator{\spread}{spread}

\newcommand{\m}{\mathfrak{m}}
\newcommand{\p}{\mathfrak{p}}

\newcommand{\cF}{\mathcal{F}}
\newcommand{\cO}{\mathcal{O}}
\newcommand{\cK}{\mathcal{K}}

\DeclareMathOperator{\edim}{edim}

\DeclareMathOperator{\gr}{gr}
\DeclareMathOperator{\Kos}{Kos}

\newcommand{\q}{\mathfrak q}

\newcommand{\ma}{\mathfrak a}
\renewcommand{\k}{K}

\title{Regularity is bounded on a quasi-excellent Noetherian scheme}

\author[De Stefani]{Alessandro De Stefani}
\address{Dipartimento di Matematica, Universit{\`a} di Genova, Via Dodecaneso 35, 16146 Genova, Italy}
\email{alessandro.destefani@unige.it}

\author[Jeffries]{Jack Jeffries}
\address{Department of Mathematics, University of Nebraska, Lincoln, NE 68588-0130, USA}
\email{jack.jeffries@unl.edu}

\author[KC]{Nawaj KC}
\address{Department of Mathematics, University of Utah, Salt Lake City, UT 84112, USA}
\email{nawaj.kc@utah.edu}

\author[N\'u\~nez-Betancourt]{Luis N\'u\~nez-Betancourt}
\address{Centro de Investigaci\'on en Matem\'aticas, Guanajuato, Gto., M\'exico}
\email{luisnub@cimat.mx}

\keywords{Betti numbers, local-to-global principle, derived category, Frobenius pushforward, generators of derived category. }
\subjclass[2020]{Primary:  	13D02,  	13A30. Secondary: 13D09, 13A35,  	14A15}

\begin{document}
\begin{abstract} 
    A point of a scheme has an associated tangent cone, the spectrum of a standard graded algebra encoding the local singularity. Its homological complexity can be measured by its graded Betti table: a matrix that records a part of the structure of its graded, minimal free resolution over a polynomial ring. A natural question is whether the homological complexity of the tangent cones varies arbitrarily across a scheme. In this paper, we show that this is not the case for a quasi-excellent Noetherian scheme; over such schemes, only finitely many graded Betti tables can occur. More generally, we show that a coherent sheaf over a quasi-excellent Noetherian scheme admits finitely many graded Betti tables, and that the constancy loci for the graded Betti table are constructible. As an immediate consequence, regularity is bounded on a quasi-excellent Noetherian scheme. 
\end{abstract}
\maketitle

    Given a scheme $X$, let $\mathcal{O}_{X,x}$ denote the local ring of $X$ at $x$ with the residue field $k(x)$. If $X$ is locally Noetherian, then the associated graded ring of $\mathcal{O}_{X,x}$ at its maximal ideal, denoted $\gr \mathcal{O}_{X,x}$, is a standard graded $k(x)$-algebra, and admits a minimal graded presentation as a quotient of a standard graded polynomial ring $S_x$ over $k(x)$. For a coherent sheaf $\mathcal{F}$ on $X$, $\gr \cF_x$ is a graded module over $S_x$. The 
    \emph{graded Betti table} of $\cF$ at $x$, denoted $\beta_{\cF_x}$, is the matrix of graded Betti numbers of $\gr \cF_x$ as an $S_x$-module:
    \[ (\beta_{\cF_x})_{i,j} = \dim_{k(x)}[\Tor_i^{S_x}(k(x), \gr \cF_x)]_j\,.\]
 
   This is an invariant of the pair $(\cF, \cO_X)$ at $x \in X$. 
   The main result of this note is that, for a coherent sheaf over any quasi-excellent Noetherian scheme, only finitely many graded Betti tables occur.

    \begin{introthm}\label{main-ertheorem}
        Let $\cF$ be a coherent sheaf on a quasi-excellent Noetherian scheme $X$.  Then $\{ \beta_{\cF_x} \, | \, x\in X\}$ is finite. Moreover, the loci on which $\beta_{\cF_x}$ is constant are constructible in~$X$.
    \end{introthm}

  We note that we make no assumptions about the finiteness of the Krull dimension of $X$. In fact, the main result holds under the assumptions that $X$ is Noetherian and satisfies the J$_2$-property (i.e., for any scheme of finite type over $X$, the singular locus is closed). 
  
  In the theorem above, the main case of interest is when $\cF$ is the structure sheaf $\cO_X$ on $X$. As its immediate corollary, we may deduce that any invariant of the local ring at $x$ that is determined or bounded by the Betti table of the local ring is also globally bounded as we range over all $x \in X$. Below, $\reg(-)$, $\spread(-)$, and $\codepth(-)$ denote the Castelnuovo-Mumford regularity, spread, and codepth of a local Noetherian ring; see \textsection\ref{prelim}.

    \begin{introcor} \label{maintheorem}
        If $X$ is a quasi-excellent Noetherian scheme, then the sets \[\{ \reg(\mathcal{O}_{X, x}) \, | \, x \in X\}, \quad
\{\spread(\mathcal{O}_{X, x})\, | \, x \in X \}, \quad \text{and} \quad \{\codepth(\mathcal{O}_{X, x})\, | \, x \in X \}\]
are finite.
    \end{introcor}

Using a construction of Hochster \cite[Proposition 1]{Hochster}, one can produce finite-dimensional Noetherian affine schemes for which each of the functions above is unbounded; such rings fail the J$_2$-property. We also note that, even if $X$ is of finite-type over a field, $\reg(\mathcal{O}_{X, x})$ and $\spread(\mathcal{O}_{X, x})$ can fail to be semicontinuous; see \cref{example} below.

A motivation and consequence of this result is a simpler proof to a recent result regarding the generation of $\D^b(R)$  for F-finite rings \cite{BILMP}; see  \cref{generationsection}. The proof we give was known to the authors of \emph{ibid.}; however, the global bound on spread was the missing ingredient.

\begin{introcor}[{\cite{BILMP}}]
    If $R$ is an F-finite Noetherian ring of prime characteristic $p > 0$ and $F$ denotes its Frobenius endomorphism, then there exists an $e \geq 1$ such that $F^e_*R$ is a generator of $\D^b(R)$. 
    \end{introcor}

A key input implicit in our argument for Theorem~\ref{main-ertheorem} is Hironaka's theory of normal flatness \cite[Chapter~II]{Hironaka}.

\section{Preliminaries.} \label{prelim}

\subsection{Betti tables and associated graded rings}

Given a commutative ring $R$, an ideal $I \subseteq R$, and an $R$-module $M$, we let $\gr_{I}(M)$ denote the associated graded module of $M$ at the ideal $I$. That is,
 \[ \gr_{I}(M) = \bigoplus _{i \geq 0} I^iM/I^{i+1}M.\]
In the case $M=R$, one has that $\gr_{I}(R)$ is a graded $R/I$-algebra, generated in degree one, and $\gr_{I}(M)$ is a graded $\gr_{I}(R)$-module.

When $(A, \m)$ is a local ring, we may simply write $\gr(A)$ and $\gr(M)$, omitting the maximal ideal $\m$ from the notation. For a standard graded algebra $R$ over a ring $A$ with $[R]_1$ a free \mbox{$A$-module}, by a \emph{minimal graded presentation} for $R$ over $A$ we mean a standard graded polynomial ring $S=A[x_1,\dots,x_n]$ over~$A$ with a surjective graded $A$-algebra homomorphism $S\twoheadrightarrow  R$ that is an isomorphism in degrees $0$ and $1$.

For a finitely generated graded module $M$ over a standard graded algebra $R$ over a field $k$ we define the \emph{graded Betti table} of $M$ to be the matrix $\beta^R(M)$ with entries \[ \beta^R(M)_{i,j} = \dim_k [\Tor_i^S(M,k)]_j\] where $S\twoheadrightarrow R$ is a minimal graded presentation. Equivalently, $\beta^R(M)_{i,j}$ is the minimal number of generators in degree $j$ and homological degree $i$ in a minimal free resolution for $M$ over $S$.

For a finitely generated $A$-module $N$ over a Noetherian local ring $A$, we define its \emph{graded Betti table} to be $\beta^A(N):= \beta^{\gr(A)}(\gr(N))$. We simply write $\beta(N)$ when $A$ is clear from context. In particular, for a coherent sheaf $\cF$ on a locally Noetherian scheme $X$ and a point $x\in X$, we have $\beta(\cF_x) = \beta(\gr(\cF_x))$ as in the introduction.

\subsection{Regularity.} \label{reg}
    Let $(A, \m, k)$ be a Noetherian local ring and let $G = \gr(A)$ denote its associated graded ring. The regularity of $A$ is defined as \[ \reg(A) = \sup\{ i+j\, | \, \H^i_{G_+}(G)_j\not= 0 \text{ for some }i\}, \]
    where $G_+ = \bigoplus_{i>0} \m^i/\m^{i+1}$ denotes the unique homogeneous maximal ideal of $G$. Alternatively, the regularity $\reg(A)$ is a measure of the size of the graded Betti table of $A$. More precisely,
     \[ \reg(A) = \sup\{j \, | \, \beta_{i,i+j}(A) \not= 0 \text{ for some } i\},
     \]
     Note that, in particular, the latter description of $\reg(A)$ in terms of Betti numbers of $G$ does not depend on the choice of a minimal presentation of $G$.

\subsection{Spread.} \label{spread}

A related invariant we attach to a local ring is its \textit{spread}. This invariant was introduced by Avramov--Iyengar--Miller \cite[3.8]{AIM}. Given a local ring $(A, \m, k)$, let $\cK$ denote the Koszul complex on the minimal generators of $\m$ which, up to isomorphism, is independent of the choice of a minimal generating set; see \cite[Lemma 3.1]{AIM}. We have a short exact sequence  \[0 \longrightarrow J \longrightarrow \cK \longtwoheadrightarrow \H_0(\cK) = k \longrightarrow 0,\] 
where $\cK \twoheadrightarrow  k$ is a surjection of dg-algebras, and the kernel $J$ is a dg ideal of $\cK$. The spread of $A$ is defined as \[ \spread(A) = \inf\{n \, | \, \H(J^i) = 0 \text{ for all } i \geq n\}.  \] That is, we have a quasi-isomorphism of $A$-complexes $\cK/J^s \simeq \cK$ if $s \geq \spread(A)$; this is the key to the proof of Theorem~\ref{theoremDbR}.

The spread of $A$ is, in fact, an invariant of its associated graded ring $G = \gr (A)$. From \cite[Proposition 3.11 and Remark 3.12]{AIM}, it follows that the spread of $A$ is equal to
\[ \spread(A) = \sup\{j \, | \, \H^i_{G_+}(G)_j\not= 0 \}  + \mu(G_+),\] 
where $\mu(-)$ denotes minimal number of generators.

It was asked whether there is a global bound on the spread of $R_{\p}$ as $\p$ varies over all prime ideals of a fixed Noetherian ring $R$ \cite[3.2]{BILMP}. We establish that such a bound exists for any quasi-excellent Noetherian ring.  In principle, such a bound would follow from upper-semicontinuity of this invariant; the next example, however, shows that both regularity and spread fail to be upper-semicontinuous, in general.

\begin{example} \label{example}
       We claim that there exists a local ring $(A,\m)$, essentially of finite type over a field, and a prime ideal $\p \subsetneq \m$, such that $\reg(A_\p) > \reg(A)$ and $\spread(A_{\p}) > \spread(A)$; in particular, $x\mapsto \reg(\mathcal{O}_{X,x})$ and $x\mapsto \spread(\mathcal{O}_{X,x})$ are not upper-semicontinuous on $X = \Spec\,  A$.  
        Let $k$ be any field and consider the local ring \[ A = \frac{k[x, y,z, w]_{(x, y, z, w)}}{(x^3w^2-y^5, xy^3w - z^5)} \] with the maximal ideal $\m = (x, y, z, w)$. Since $\gr(A)$ is a complete intersection of codimension two cut out by equations of degree five, we have $\reg(A) = 5 + 5 -2 =8$ and $\spread(A) = 10.$

        We now consider $\p = (x, y, z) \in \Spec\, A$. Localizing $A$ at $\p$, we have \[ A_{\p} \cong \frac{k(w)[x, y, z]_{(x,y,z)}} {(w^2x^3-y^5, wxy^3-z^5)}.\]
        A standard Gr{\"o}bner-bases calculation shows that
        \[
        \gr_{\p A_\p}(A_\p) \cong \frac{k(w)[x,y,z]}{(x^3,xy^3,x^2z^5,y^{11}-xz^{10})},
        \]
        and thus $\reg(A_\p) \geq 10> 8=\reg(A)$. Futhermore, since $0 \ne x^2y^2z^4 \in \H^0_{(x, y, z)}(\gr_{\p A_{\p}}(A_{\p}))$, we have that $\spread(A_{\p}) \geq 11 > 10 = \spread(A).$  
    \end{example}

\subsection{Codepth.} The codepth of a local ring $(A, \m)$ is defined as \[ \codepth A = \edim A - \depth A,\]where $\edim A$ is equal to the minimal number of generators of the maximal ideal of $A$. For a locally Noetherian scheme $X$, one defines \[ \codepth(X) = \sup \{ \codepth(\mathcal{O}_{X,x})\, | \, x \in X \}.\]This invariant plays a central role in the proof that $F^e_* R$ generates $\D^b(R)$ \cite{BILMP} and it is shown that the codepth is finite for a Noetherian F-finite scheme \cite[Proposition 3.4]{BILMP}. We show that codepth is finite for all Noetherian quasi-excellent schemes.

\section{Main results}

 The key input in this paper is the following theorem. As corollaries, one obtains proofs for  \Cref{main-ertheorem} and \Cref{maintheorem}.

 \begin{theorem} \label{mainlemma}
     Suppose that $R$ is a Noetherian ring and $\p \in \Spec\, R$ is such that the regular locus of $R/\p$ is an open set. If $M$ is a finitely generated $R$-module, then there exists $f \in R \smallsetminus \p$ such that for all $\q \in \V(\p) \smallsetminus \V(f)$ the graded Betti tables $\beta(M_\q)$ are constant. 
      \end{theorem}

Our proof of this theorem requires two main ingredients.

\begin{chunk}
    The first main ingredient comes from Hironaka's theory of normal flatness \cite[Chapter~II]{Hironaka}. We assume the following setup: 
    
 \begin{setup}\label{setupHironaka} Let  $R$ be a Noetherian ring and $I \subseteq J$ be ideals of $R$ such that $\overline{J}=J/I$ is a complete intersection ideal in $\overline{R} = R/I$. We denote $\k = R/J = \overline{R}/\overline{J}$.
 \end{setup}

  Under Setup~\ref{setupHironaka}, we have a natural map of graded algebras \[ \gr_{J}(R) \longrightarrow \gr_{\overline{J}} (\overline{R}),\] and we are interested in the case when it admits a section. The assumption that $\overline{J}$ is a complete intersection ideal implies that $\gr_{\overline{J}}(\overline R)$ is a symmetric algebra over $\k$ generated by the images of (any choice of) $f_1, \ldots, f_t$ in $J$ that generate the ideal $\overline{J}$. Thus there is a $\k$-algebra section \[ \gamma: \gr_{\overline{J}}(\overline{R}) \longrightarrow \gr_J(R),\]
    given by $\overline{f_i} +J^2 \mapsto f_i +J^2$. Note that this map depends on the choice of the lift of the generators.

    \begin{example}\label{ex-setup}
         Suppose that $(R,\m)$ is a Noetherian local ring and let $\ma$ be some ideal such that $R/\ma$ is regular. Then $\ma \subseteq \m$ satisfy the conditions of Setup~\ref{setupHironaka}. 
    \end{example}

     Under Setup~\ref{setupHironaka}, for any $R$-module $M$, we have inclusions $I^nM \subseteq J^n M$ for each $n \geq 0$. This induces the following map of graded $\k$-modules
    \[ \alpha_M: \gr_I(M)/J \gr_I (M) \longrightarrow \gr_{J}(M). \] 
    Using the action of $\gr_J(R)$ on $\gr_J(M)$, we then define the map \[ \alpha_M \otimes \gamma: \gr_{I}(M)/J\gr_{I} (M) \otimes_{\k} \gr_{\overline J}(\overline R) \longrightarrow \gr_J(M) \] where $(\alpha_M \otimes \gamma)(a \otimes b) = \alpha_M(a) \gamma(b)$. Equipping the source with the total grading of the tensor product, this is a map of graded $\k$-modules. 
    When $M=R$, the map $\alpha_R \otimes \gamma$ is a homomorphism of graded $K$-algebras. In general, $\gr_{I}(M)/J\gr_{I} (M) \otimes_{\k} \gr_{\overline J}(\overline R)$ is a graded $G:=\gr_{I}(R)/J\gr_{I} (R) \otimes_{\k} \gr_{\overline J}(\overline R)$-module, and considering $\gr_J(M)$ as a $G$-module by restriction of scalars through $\alpha_R \otimes \gamma$, the map $\alpha_M \otimes \gamma$ is a homomorphism of $G$-modules.

    The following proposition is essentially due to Hironaka \cite[Proposition II.1 p.~184]{Hironaka}, which handles the $M = R$ case. Building on his idea, we give somewhat more streamlined proof and generalize it to all finitely generated modules. 
    
    \begin{proposition}[Hironaka]\label{PropHironaka}
        In Setup~\ref{setupHironaka}, if $M$ is a finitely generated $R$-module such that $\gr_I(M)$ is free $\overline{R}$-module, then \[ \alpha_M \otimes \gamma:  \gr_{I}(M)/J \gr_{I} (M) \otimes_{\k} \gr_{\overline J}(\overline R) \longrightarrow \gr_J(M) \] is an isomorphism.  
    \end{proposition}
        \begin{proof}
            We have $J = I +(f_1, \ldots, f_t)$ where we may assume $t \geq 1$. Our proof is by induction on $t$ and we begin by assuming that the assertion holds when $t = 1$. Let $J' = I + (f_1)$ and $\k'= R/J'$. We then have \[ \gr_{I}(M)/ f_1\gr_I(M) \otimes_{\k'} \gr_{\overline{J'}}(\overline R) \cong \gr_{J'}(M).\]  Since $\gr_{I}(M)/f_1\gr_{I}(M)$ remains flat as a $\overline{R}/f_1\overline{R} = R/J'$-module, along with the isomorphism above, we deduce that $\gr_{J'}(M)$ is a flat $R/J'$-module. Now the inductive hypothesis kicks in to give us the desired result. Therefore it suffices to prove the case when $J = I + (f)$. 
        
            To that end, it is enough to show that $(\alpha_M \otimes \gamma)_n$ is a bijective map for each $n \geq 1$. Consider the following commutative diagram \[ \xymatrix{  \bigoplus\limits_{i+j=n} I^i M \otimes_R (f)^{j} \ar[d]_{\mu} \ar[dr]^{\nu} \ar[r]^-{\eta}& J^n M \ar[d]^{\pi} \\  \bigoplus\limits_{i+j=n} \displaystyle \left[\frac{\gr_I(M)}{J \gr_I(M)}\right]_i \otimes_{\k} [\gr_{\overline J}(\overline R)]_j \ar[r]_-{}  & [\gr_J(M)]_n }\]
where the maps $\mu$ and $\pi$ are induced by the natural surjections and $\eta$ is induced by the action. The maps $\mu$ and $\nu$ are surjective, so it suffices to show that $\ker(\nu) \subseteq \ker(\mu)$. Fix $h\in \ker(\nu)$. Suppose that $h = h_m + \cdots + h_n$ where $h_i \in I^iM \otimes_R (f)^{n-i}$ and $m \geq 0$. Since $\nu(h) =0$, we must have $\eta(h) \in J^{n+1}M = \sum_{i=0}^{n+1}(f)^{n+1 -i}I^iM$.

We claim that we can find an expression $\eta(h) = g_m + \cdots+g_{n+1}$ where $g_i \in (f)^{n+1 -i}I^iM$. Let \begin{equation} \label{eqn}
    \eta(h) = g_r + g_{r+1} + \cdots + g_{n+1}
\end{equation} be any expression for $\eta(h)$ with $g_i\in (f)^{n+1 -i}I^iM$ and  $g_r \neq 0$. If $r\geq m$, there is nothing to do. Otherwise, suppose that $r<m$. It suffices to show that there is an expression $\eta(h) = g'_r + \cdots + g'_{n+1}$ with $g'_i\in (f)^{n+1 -i}I^iM$ and $g'_r=0$, as we can obtain an expression of the desired form proceeding inductively.

To this end, for $i=0,\dots,n$, let $\{b_{i,1},\dots,b_{i,t_i}\}\subseteq I^i M$ be elements such that their images mod $I^{i+1}M$ form a free $\overline{R}$-module basis for $I^i M/I^{i+1}M$, and thus also an $R$-module generating set for $I^i M$. Let $h_i = \sum_k b_{i, k} \otimes r_{i, k}f^{n-i}$ and $g_i = \sum_{k} f^{n+1-i}s_{i,k}b_{i,k}$ for all $i$. Modulo $I^{r+1}M$, the equality \cref{eqn} can be written as $ \sum_{k} s_{r, k}f^{n+1-r}b_{r, k} = 0$, and therefore we must have $s_{r,k}f^{n+1-r} =0 $ over $\overline R$ for all $k$. Since $r < n+1$, $f^{n+1-r}$ is a nonzerodivisor in $\overline{R}$. This implies that $s_{r,k} \in I$ for all $k$. Rewriting $s_{r,k}b_{r,k}$ in terms of the basis of $I^{r+1}M$, we have produced a different expression of $\eta(h) = g'$ with $g'_r = 0$. This completes the justification of the claim.

 Given the claim, we thus set $r = m$ and consider the equality \cref{eqn} modulo $I^{m+1}M$. Arguing similarly as above, we deduce $r_{m, k} = s_{m,k}f$ modulo $I$, and thus $r_{m,k} \in I+(f) = J$ for all $k$. In particular, $\mu(h_m) = 0$. Proceeding inductively in a similar manner, we thus get $\mu(h)=0$.
\end{proof}

From this we deduce the following key corollary.

\begin{corollary} \label{cor-family}
    Suppose that $R$ is a Noetherian ring and $\p \in \Spec\, R$ is such that the regular locus of $\overline{R}:=R/\p$ is an open set. If $M$ is a finitely generated $R$-module, there exists $f \in R \smallsetminus \p$ such that for all $\q \in \V(\p)\smallsetminus \V(f)$, there is an isomorphism of graded modules \[ \kappa(\q) \otimes_{\overline{R}_f} \gr_{\p}(M)_f \otimes_{\kappa(\q)} Q \cong \gr_{\q R_{\q}}(M_\q),\] where $Q:= \gr_{\overline{\q}}(\overline{R}_\q)\cong \kappa(\q)[x_1,\dots,x_s]$ for $s=\mathrm{height}(\q/\p)$. Furthermore, for all such $\q$, \[\beta(\kappa(\q) \otimes_{\overline{R}_f} \gr_{\p}(M)_f ) = \beta(M_\q).\]
\end{corollary}

\begin{proof}
                Since $\text{Reg}(\overline{R})$ is a nonempty open set, there is some $g\in R\smallsetminus \p$ such that $\overline{R}_{g}$ is regular. By applying the generic freeness lemma (e.g., \cite[(22.A)~Lemma 1]{MatsumuraCA})
                to $\gr_{\p}(R)$ and $\gr_\p(M)$, there exists some $h \in R\smallsetminus \p$ such that
                $\gr_{\p}(R)_{h}$ and $\gr_\p(M)_h$ are free $\overline{R}_{h}$-modules. Set ${f=gh}$. Then for each $\q\in U:= \V(\p)\smallsetminus \V(f)$, we have that $\overline{R}_\q$ is regular, and both $\gr_{\p R_\q}(R_\q)$ and $\gr_{\p R_\q}(M_\q)$ are free $\overline{R}_{\q}$-modules. The isomorphism $Q\cong \kappa(\q)[x_1,\dots,x_s]$ follows from Example~\ref{ex-setup}.
               
		For $\q\in U$, note that 
                \[ \kappa(\q) \otimes_{\overline{R}_f} \gr_{\p}(M)_f \cong \frac{\gr_{\p R_{\q}}(M_\q)}{\q \gr_{\p R_{\q}}(M_\q)},\] where $\kappa(\q)= R_\q/\q R_\q$. Due to Proposition~\ref{PropHironaka}, for $\q\in U$ we have an isomorphism of graded $\kappa(\q)$-algebras:
                \[  \gr_{\q R_\q}(R_\q) \cong \kappa(\q) \otimes_{\overline{R}_f} \gr_{\p}(R)_f  \otimes_{\kappa(\q)} Q  \]
and a compatible isomorphism of graded modules         
                \[  \gr_{\q R_\q}(M_\q) \cong \kappa(\q) \otimes_{\overline{R}_f} \gr_{\p}(M)_f  \otimes_{\kappa(\q)} Q.  \]
                
Now, let $S\twoheadrightarrow \kappa(\q) \otimes_{\overline{R}_f} \gr_{\p}(R)_f$ be a minimal graded presentation as a graded $\kappa(\q)$-algebra, and let $P$ be a minimal free resolution of $\kappa(\q) \otimes_{\overline{R}_f} \gr_{\p}(M)_f$ as a graded $S$-module. Then $S \otimes_{\kappa(\q)} Q \twoheadrightarrow \kappa(\q) \otimes_{\overline{R}_f} \gr_{\p}(R)_f  \otimes_{\kappa(\q)} Q$ is a minimal graded presentation as a graded \mbox{$\kappa(\q)$-algebra}, and $P\otimes_{\kappa(\q)} Q$ is a minimal free resolution of $\kappa(\q) \otimes_{\overline{R}_f} \gr_{\p}(M)_f  \otimes_{\kappa(\q)} Q$ over $S \otimes_{\kappa(\q)} Q$. Thus, we have
\[ \begin{aligned}\hspace{1in} \beta( \gr_{\q R_\q}(M_\q) )_{i,j} &= \beta(\kappa(\q) \otimes_{\overline{R}_f} \gr_{\p}(M)_f  \otimes_{\kappa(\q)} Q)_{i,j} \\
&=\dim_{\kappa(\q)} \left[ \H_i \left( (P\otimes_{\kappa(\q)} Q) \otimes_{S\otimes_{\kappa(\q)} Q} \kappa(\q) \right) \right]_j \\
&=\dim_{\kappa(\q)} \left[ \H_i \left( P \otimes_{S} \kappa(\q) \right) \right]_j  \\
&= \beta(\kappa(\q) \otimes_{\overline{R}_f} \gr_{\p}(M)_f )_{i,j}. \hspace{1.85in} \qedhere \end{aligned}\]
           \end{proof}
            \end{chunk}

    \begin{chunk}
    To finish our proof of \cref{mainlemma} we use the fact that the Betti tables of the fibers of finitely generated modules over standard graded Noetherian algebras over a domain vary in a constructible manner. For related results, see \cite{BG}.
    \end{chunk}

    \begin{lemma}\label{lem-family}
        Suppose that $A$ is a Noetherian domain and $T$ is a standard graded Noetherian $A$-algebra with $[T]_1$ free as an $A$-module. Let $L$ be a finitely generated $T$-module. There exists a nonzero $f \in A$ such that for all $\p \in \Spec\, A \smallsetminus \V(f)$,  we have $\beta(\kappa(\p) \otimes_A L)$ is constant.
    \end{lemma}
        \begin{proof}
            We have $T_0 = A$ and suppose that $\mathrm{rank}_A(T_1) = e$. Consider a minimal surjection $S = A[x_1, \ldots, x_e] \twoheadrightarrow T$. For $\p \in \Spec\, A$, denote $\kappa(\p) = A_\p/\p A_\p$, and let $S(\p) = \kappa(\p) \otimes_A S$, $T(\p) = \kappa(\p) \otimes_A T$, and $L(\p)=\kappa(\p) \otimes_A L$.  Note that for any $\p \in \Spec\, A $, we have $S(\p) \twoheadrightarrow T(\p)$ is also a minimal surjection, and $L(\p)$ is thus an $S(\p)$-module via this map. We view $A$ as an $S$-module via the isomorphism $S/(x_1, \ldots,x_e)S \cong A$. Since $x_1,\ldots,x_e$ forms an $S$-regular sequence, the minimal $S$-free resolution of $A$ is given by $\cK = \Kos^S(x_1, \ldots, x_e)$, the Koszul complex on $x_1, \ldots, x_e$ over $S$. Now observe that 
            \[ L(\p) \otimes_S \cK \simeq L(\p) \lotimes_{S(\p)} S(\p) \otimes_{S} \cK \simeq L(\p) \lotimes_{S(\p)} \kappa(\p),\]
            where for the second quasi-isomorphism we note that $S(\p) \otimes_S K$ is a minimal $S(\p)$-free resolution of $\kappa(\p)$. Therefore, setting $\cK^L = L \otimes_S \cK$, for all $\p \in \Spec\, A$ we have that \[ \beta_{i, j}(\kappa(\p) \otimes_A L) = \dim_{\kappa(\p)}[\H_i(\kappa(\p) \otimes_A \cK^L)]_j.\]

            By applying the generic freeness lemma (e.g., \cite[(22.A)~Lemma 1]{MatsumuraCA}), there exists an element $0 \ne f \in A$ such that all the boundaries and nonzero homology modules of $\cK^L_f=(\cK^L)_f$ are free $A_f$-modules. Consequently, the short exact sequences 
            \[ 0 \longrightarrow B_i(\cK^L_f) \longrightarrow Z_i(\cK^L_f) \longrightarrow \H_i(\cK^L_f) \longrightarrow 0, \] \[ 0 \longrightarrow Z_i(\cK^L_f) \longrightarrow (\cK^L_i)_f \longrightarrow B_{i-1}(\cK^L_f) \longrightarrow 0\] are in fact all split exact sequences of $A_f$-modules, for all $i$. Thus, for any $\p \in \Spec\, A \smallsetminus \V(f)$ we may conclude that  \[ \H_* (\kappa(\p) \otimes_A \cK^L) \cong \H_*(\kappa(\p) \otimes_{A_f} \cK^L_f) \cong  \kappa(\p) \otimes_{A_f} \H_*(\cK^L_f).\]
                Since on the right hand side the homology modules are free, their $\kappa(\p)$-vector space dimension is unchanged as $\p$ varies.  
                This completes our proof.
        \end{proof}

\begin{proof}[Proof of \cref{mainlemma}]
    First we apply \cref{cor-family} to deduce the existence of $g \in R \smallsetminus \p$ such that for all $\q \in \V(\p) \smallsetminus \V(g)$ we have $\beta(\kappa(\q) \otimes_{(R/\p)_g} \gr_{\p}(M)_g ) = \beta(M_\q).$ Then we apply \cref{lem-family} with $A := (R/\p)_g$, $T = \gr_{\p}(R)_g$, and $L = \gr_{\p}(M)_g$ to deduce the existence of $h \in R$ whose image in $A$ is nonzero and for all $\q \in \Spec \, A \smallsetminus \V(h)$, we have $\beta(\kappa(\q) \otimes_A L) = \beta(M_\q)$ is constant. Identifying $\Spec\, A = \V(\p) \smallsetminus \V(g)$, we have $\Spec\, A \smallsetminus \V(h) = \V(\p) \smallsetminus \V(gh)$. Then setting $f = gh \in R \smallsetminus \p$, our proof is complete. 
\end{proof}

\begin{chunk}
    If $X$ denotes a Noetherian scheme with the J$_2$-property, then we can choose an affine cover, \[ X = \bigcup _{i = 1}^n\Spec\, R_i,\]where each $R_i$ is a Noetherian ring with the J$_2$-property; this means that every finite $R_i$-algebra has an open regular locus \cite[\href{https://stacks.math.columbia.edu/tag/07P7}{Definition 07P7}]{stacks-project}. 
\end{chunk}

We now give the proof to our main theorem. The argument given below is inspired by work of Bennett \cite[Chapter III, Remark 1.3]{Bennett}.

\begin{proof}[Proof of \cref{main-ertheorem}]
    It suffices to show that we can write $X$ as a finite union of locally closed subsets $W_{i,j}$ such that $\beta(M_\q)$ is constant on each~$W_{i,j}$. First, we prove the case when ${X = \Spec\, R}$ for $R$ a Noetherian ring with the J$_2$-property. Then $M = \Gamma(\cF, X)$ is a finitely generated $R$-module satisfying $\cF_x \cong M_{\p_x}$ for each $x \in X$, where $\p_x$ is the prime ideal of $R$ corresponding to $x$. 
    
    We build a sequence of closed subsets $Z_i$ and locally closed subsets $W_{i,j}$ inductively as follows. We set $Z_0=\Spec\, R$. Given~$Z_i$, if $Z_i=\varnothing$, we set $Z_{i+1}=\varnothing$. If $Z_i\neq \varnothing$, 
    let $\p_1,\dots,\p_t\in Z_i$ be the prime ideals defining the irreducible components of $Z_i$.
  Then we apply \cref{mainlemma} to each $\p_j$ to obtain an nonempty set $W_{i.j}$ that is a dense open subset of $\V(\p_j)$ on which $\beta(M_\q)$ is constant. Then $Z_{i+1} := \bigcup_j \big(\V(\p_j) \smallsetminus W_{i,j}\big)$ is a proper closed subset of $Z_i$, and in particular closed in $X$. In this way we obtain a
descending chain of closed subsets \[ Z_1 \supseteq Z_2 \supseteq \cdots \] which must stabilize since $R$ is Noetherian. Hence, $Z_i=\varnothing$ for $i\gg 0$. Thus, there are finitely many sets $W_{i,j}$, which together form a locally closed cover of $X$ by construction. Thus, $\{W_{i,j}\}$ forms the collection we seek, which completes the proof in the affine case.

       Finally, for the general case, it remains to observe that we can choose an affine cover \[ X = \bigcup_{i=1}^n \Spec\, R_i\] where each $R_i$ is a Noetherian ring with the J$_2$-property. For any point $x \in X$, choosing an index $j$ such that $x \in \Spec\, R_j$ yields a canonical isomorphism of stalks $\cF_x \cong (M_j)_x$, where $M_j$ is the module of sections of $\cF$ on $\Spec\, R_j$. We then take for each affine $R_i$ a locally closed cover as above. This completes the proof.
       \end{proof}

\begin{proof}[Proof of \cref{maintheorem}]
Since there are finitely many graded Betti tables of $\cO_{X,x}$ as we range over all points $x$ of $X$, there are finitely many values of regularity and spread achieved by the local rings of the scheme; this follows from the definition of regularity and spread given above. Thus \cref{main-ertheorem} suffices for regularity and spread. As for codepth, note that for $x \in X$, \[ \codepth \mathcal{O}_{X, x}  = \pdim_{Q(x)} \widehat{\mathcal{O}_{X,x}} \leq \pdim_{\gr Q(x)} \gr \widehat{\mathcal{O}_{X,x}},\] where $\widehat{\mathcal{O}_{X,x}}$ denotes the completion of $\mathcal{O}_{X,x}$  at the maximal ideal and $Q(x)$ denotes the regular local ring that minimally surjects onto this complete local ring. The equality follows by the fact that $\edim $ and $\depth$ are both invariant under completion and then by applying the Auslander--Buchsbaum formula. For the inequality, see for instance \cite[Example 1]{Froberg}. Finally, noting that the graded Betti table is also invariant under completion, we are done by \cref{main-ertheorem}.
\end{proof}

\subsection{An application to generation of derived categories.} \label{generationsection}

Suppose that $R$ is a Noetherian ring and let $\D^b(R)$ denote the bounded derived category of $R$. Given an object $E \in \D^b(R)$, the \textit{thick} subcategory generated by $E$, denoted $\thick(E)$, is the smallest triangulated subcategory of $\D^b(R)$ containing $E$, closed under taking summands.

We say $E \in \D^b(R)$ generates $\D^b(R)$ if $\thick(E) = \D^b(R)$. This  can be checked locally: an object $G$ is a generator of $\D^b(R)$ if and only if for all $\p \in \Spec\, R$ the residue field $\kappa(\p)$ of the local ring $R_\p$ is in $\thick(G_\p)$ \cite[Corollary 2.10]{BILMP}. With this result, we now give a simpler proof to the following result regarding the generation of $\D^b(R)$ \cite{BILMP}, which we learned from the authors of \emph{ibid.}

\begin{theorem}[\cite{BILMP}] \label{theoremDbR}
   If $R$ is an F-finite Noetherian ring of prime characteristic $p > 0$ and $F$ denotes its Frobenius endomorphism, then there exists an $e \geq 1$ such that $F^e_*R$ is a generator of $\D^b(R)$. 
\end{theorem}

\begin{proof}
Since $R$ is F-finite, $X = \Spec\, R$ is an excellent scheme \cite{Kunz}. Therefore, \cref{maintheorem} implies that $\spread(X):= \sup \{\spread(R_{\p})\, | \, \p \in X \}$ is finite. Fix some positive integer ${e > \log_p(\spread(X))}$. From our discussion above, it suffices to show that for any $\p \in X$, we have $\kappa(\p) = R_\p/\p R_\p$ is in $\thick(F^e_*R_\p).$

To this end, let $(A, \m, k)$ denote the local ring $R_\p$. Let $\cK$ be the Koszul complex on a set of minimal generators of $\m$ and consider the dg ideal $J$ that fits in the short exact sequence $0 \to J \to \cK \to k \to 0$; see \cref{spread}. Since $s \geq \spread(X)$, we have that \[ \cK \simeq \cK/J^s.\] Let $F$ denote the Frobenius endomorphism on $A$. Since $\m^s (\cK/J^s) = 0$, the map 
\[
A \stackrel{F^e}{\longrightarrow} A \longrightarrow \cK \longrightarrow \cK/J^s\]
 factors through $k$. Thus,

   \[k \in \thick(F^e_*(\cK/J^s)) = \thick(F^e_*(\cK)) \subseteq \thick(F^e_*(A)). \qedhere\]

\end{proof}

\subsection*{Acknowledgements} We thank Shiji Lyu for suggesting the constructibility statement in Theorem~\ref{main-ertheorem} and a useful discussion on Lemma~\ref{lem-family}. We thank Josh Pollitz for suggesting us to consider the graded Betti tables of finitely generated modules.  We thank Giulio Caviglia, Aldo Conca, Elo\'isa Grifo, Srikanth Iyengar, Thomas Polstra, and Maria Evelina Rossi for additional helpful discussions.

The first author was partially supported by the MIUR Excellence Department Project CUP D33C23001110001, PRIN 2022 Project 2022K48YYP, and by INdAM-GNSAGA. The second author was partially supported by NSF CAREER Award DMS-2044833 and SECIHTI grant CF-2023-G-33. The fourth author was partially supported by SECIHTI Grants CBF 2023-2024-224, CF-2023-G-33 and CBF-2025-I-673.

\newcommand{\etalchar}[1]{$^{#1}$}


\begin{thebibliography}{BIL{\etalchar{+}}26}
















\bibitem[AIM06]{AIM}
Luchezar~L. Avramov, Srikanth Iyengar, and Claudia Miller.
\newblock Homology over local homomorphisms.
\newblock {\em Amer. J. Math.}, \textbf{128}(1), 23--90, 2006.
















\bibitem[Ben70]{Bennett}
Bruce~Michael Bennett.
\newblock On the characteristic functions of a local ring.
\newblock {\em Ann. of Math. (2)}, \textbf{91}, 25--87, 1970.
















\bibitem[BG85]{BG}
Maksymillian~Boraty\'{n}ski and Silvio~Greco.
\newblock Hilbert functions and {B}etti numbers in a flat family.
\newblock {\em Ann. Mat. Pura Appl. (4)}, \textbf{142}, 277--292, 1985.
































\bibitem[BILMP26]{BILMP}
Matthew~R. Ballard, Srikanth~B. Iyengar, Pat Lank, Alapan Mukhopadhyay, and
  Josh Pollitz.
\newblock High {F}robenius pushforwards generate the bounded derived category.
\newblock {\em Forum Math. Sigma}, \textbf{14}, Paper No. e12, 2026.
































\bibitem[Fr{\"o}87]{Froberg}
Ralf Fr{\"o}berg.
\newblock Connections between a local ring and its associated graded ring.
\newblock {\em J. Algebra}, \textbf{111}(2), 300--305, 1987.
































\bibitem[Hir64]{Hironaka}
Heisuke Hironaka.
\newblock Resolution of singularities of an algebraic variety over a field of
  characteristic zero.~{I}.
\newblock {\em Ann. of Math. (2)}, \textbf{79}, 109--203, 1964.
















\bibitem[Hoc73]{Hochster}
Melvin~Hochster.
\newblock Non-openness of loci in {N}oetherian rings.
\newblock {\em Duke Math. J.}, \textbf{40}, 215--219, 1973.




















\bibitem[Kun76]{Kunz}
Ernst~Kunz.
\newblock On Noetherian rings of characteristic $p$.
\newblock {\em Amer. J. Math}, \textbf{98}, 999--1013, 1976.
































\bibitem[Mat80]{MatsumuraCA}
Hideyuki Matsumura.
\newblock {\em Commutative algebra}, volume~56 of {\em Mathematics Lecture Note
  Series}.
\newblock Benjamin/Cummings Publishing Co., Inc., Reading, MA, second edition,
  1980.
















































\bibitem[{Sta}18]{stacks-project}
The {Stacks Project Authors}.
\newblock \textit{Stacks Project}.
\newblock \url{https://stacks.math.columbia.edu}, 2018.
















\end{thebibliography}
\end{document}